\documentclass[pdflatex,11pt]{article}
\usepackage[T1]{fontenc}
\usepackage[applemac]{inputenc}
\usepackage[english]{babel}
\usepackage{geometry}
\usepackage{amsfonts,amsmath,amssymb,amsthm}
\usepackage{lmodern}
\usepackage{graphicx}
\usepackage{enumerate}
\usepackage{url}

\makeatletter
\newtheorem*{rep@theorem}{\rep@title}
\newcommand{\newreptheorem}[2]{%
\newenvironment{rep#1}[1]{%
 \def\rep@title{#2 \ref{##1}}%
 \begin{rep@theorem}}%
 {\end{rep@theorem}}}
\makeatother

\newtheorem{theo}{Theorem}[section]
\newreptheorem{theo}{Theorem}
\newtheorem{prop}[theo]{Proposition}
\newtheorem{lem}[theo]{Lemma}
\newtheorem{cor}[theo]{Corollary}
\newtheorem*{conj}{Conjecture}

\theoremstyle{definition}
\newtheorem{defi}[theo]{Definition}

\newtheorem*{ack}{Acknowledgments}

\theoremstyle{remark}

\newtheorem*{rem*}{Remark}

\newcommand\C{\mathbb C}
\newcommand\B{\mathcal B}

\newcommand\tra[1]{\beta_{#1}}
\newcommand\trak{\tra{k}}
\newcommand\trb[1]{\tilde\beta_{#1}}
\newcommand\trbk{\trb{k}}
\newcommand\tdelta{\tilde\delta}
\newcommand\atome[2]{\alpha_{#1 #2}}
\newcommand\init[1]{S(#1)}
\newcommand\final[1]{F(#1)}
\newcommand\disque[1]{D_{#1}}
\DeclareMathOperator\mcg{MCG}

\newcommand\orb{\mathcal O}

\author{Sandrine Caruso}
\title{A family of pseudo-Anosov braids whose super summit sets grow exponentially}
\date{}

\begin{document}

\maketitle


\begin{abstract}
We prove that the size of the super summit set of a braid can grow exponentially with the canonical length of the braid, even for pseudo-Anosov braids.
\end{abstract}

\section*{Introduction}

When trying to find a polynomial time solution to the conjugacy problem in braid groups, one strategy is to use Garside theory in order to associate to any given braid a certain finite subset of its conjugacy class. This subset should only depend on the conjugacy class of the given braid, and it should be effectively computable. Among the invariants of this type, the super summit set, introduced by El-Rifai and Morton in \cite{EM}, is one of the most famous. More recently, in \cite{GGM}, Gebhardt and Gonz\'alez-Meneses have introduced a new invariant, the set of sliding circuits, which is a subset of the super summit set. 

Given two elements of the braid group, the strategy for deciding if they are conjugate is then in two steps: first, for each of the two braids, find one element of the invariant subset of its conjugacy class; then, calculate the complete invariant subset of one of the two in order to check whether it contains the other element. 

%

There does exist a polynomial time algorithm which, given a braid, finds an element of its super summit set \cite{BKL}. However, determining whether two such elements are in the same super summit set seems to be much more difficult to do quickly. Indeed, in general, the size of the super summit set of a braid is not bounded by a polynomial in the length of the braid:
in \cite{GM}, Gonz\'alez-Meneses constructs a family of reducible braids whose super summit set grows exponentially, both with the length of the braid and with the number of strings. In fact, even the sets of sliding circuits of his braids grow exponentially.

One might have hoped that this bad behavior is specific to reducible braids. However, in \cite{prasolov}, Prasolov presents a family of pseudo-Anosov braids whose ultra summit set, another subset of the super summit set, grows exponentially with the number of strings of the braid. Here, inspired by the construction of Gonz\'alez-Meneses in \cite{GM}, we 
present a family of pseudo-Anosov braids with 5 strings whose super summit set grows exponentially in the length of the braids. 
\begin{reptheo}{theo:exp}
Let $k \geqslant 1$. The braid
\[\trak =(\sigma_2\sigma_1)^{3k+1}\sigma_4^{2k+2} \sigma_3\sigma_4^{2k-1}\]
is pseudo-Anosov and has a super summit set whose cardinality is at least $2^{2k-2}$.
\end{reptheo}
It is interesting to note that, on the other hand, the set of sliding circuits of this family of braids is as small as possible (Lemma \ref{lem:orbite}). This is relevant to a conjecture, well known to the specialists of the domain:
\begin{conj}
For all $n$, there exists a polynomial $P_n$ such that, for every pseudo-Anosov braid~$x$ with $n$~strings and with canonical length $\ell$, the cardinality of the set of sliding circuits of~$x$ is at most $P_n(\ell)$.
\end{conj}

\begin{ack}I would like to thank my PhD advisor Bert Wiest for his help and guidance.
\end{ack}

\section{Definitions}

\subsection{Braids and mapping class group of the punctured disk}

\begin{defi}[Mapping class group of the punctured disk]
Let $\disque n$ be the closed unit disk in $\C$, with $n$ punctures regularly spaced on the real axis. The \emph{mapping class group} of $\disque n$, denoted $\mcg(\disque n)$, is the group of the homeomorphisms of $\disque n$, modulo the isotopy relation. We also denote $\mcg(\disque n, \partial \disque n)$ the group of the homeomorphisms of $\disque n$ fixing pointwise the boundary $\partial \disque n$ of $\disque n$, modulo the isotopy relation.
\end{defi}

The Artin braid group with $n$ strings is isomorphic to the group $\mcg(\disque n, \partial \disque n)$.

Recall that the classification theorem of Nielsen and Thurston states that a mapping class $f \in \mcg(\disque n)$ is either periodic, or reducible, or pseudo-Anosov. A braid $x \in \mcg(\disque n, \partial \disque n)$ can be projected on an element of $\mcg(\disque n)$. We call \emph{Nielsen-Thurston type of $x$} the Nielsen-Thurston type of its projection. The definition of periodicity is then transformed as follows: a braid $x \in \B_n$ is periodic if and only if there exist nonzero integers $m$ and $l$ such that $x^m = \Delta^l$, where $\Delta = (\sigma_1 \cdots \sigma_{n-1})(\sigma_1\cdots\sigma_{n-2}) \cdots (\sigma_1\sigma_2)\sigma_1$. (Geometrically $\Delta$ corresponds to the half-twist around the boundary of the disk).

\subsection{Garside structure and invariants of conjugacy classes}

A general introduction to Garside theory can be found in \cite{DDGM}. We shall only recall some facts which are useful for our purposes.

While the group $\B_n$ admits the well-known presentation of groups
\[\B_n = \left< \sigma_1, \ldots, \sigma_{n-1}\ ;\ \sigma_i \sigma_{i+1} \sigma_i = \sigma_{i+1}\sigma_i \sigma_{i+1} \text{ and } \sigma_i \sigma_j = \sigma_j \sigma_i \text{ for } |i-j| \geqslant 2 \right>,\]
the \emph{monoid of positive braids} $\B_n^+$, which is embedded in $\B_n$, is defined by the same presentation, interpreted as a presentation of monoids.

For $m \leqslant n$, we denote $\Delta_m$ the element of $\B_n^+$ defined by
\[\Delta_m = (\sigma_1 \cdots \sigma_{m-1})(\sigma_1\cdots\sigma_{m-2}) \cdots (\sigma_1\sigma_2)\sigma_1\]
and we will denote $\Delta = \Delta_n \in \B_n^+$.

The pair $(\B_n^+,\Delta)$ defines what we call a Garside structure on $\B_n$. Without giving the complete definition, here are some properties of such a structure. The group $\B_n$ is endowed with an order $\preccurlyeq$ defined by $x \preccurlyeq y \Leftrightarrow x^{-1}y \in \B_n^+$. If $x \preccurlyeq y$, we say that $x$ is a \emph{prefix} of $y$. Any two elements $x,y$ of $\B_n$ have a unique greatest common prefix.

We also define $\succcurlyeq$ by $x \succcurlyeq y \Leftrightarrow x y^{-1} \in \B_n^+$. Note that $x \succcurlyeq y$ is not equivalent to $y \preccurlyeq x$.

The elements of the set $\{x \in \B_n, 1 \preccurlyeq x \preccurlyeq \Delta\}$ are called \emph{simple braids}.

\begin{defi}[left-weighting]
Let $s_1$, $s_2$ be two simple braids in $\B_n$. We say that $s_1$ and $s_2$ are \emph{left-weighted} if there does not exist any generator $\sigma_i$ such that $s_1 \sigma_i$ and $\sigma_i^{-1} s_2$ are both still simple.
\end{defi}

\begin{prop}
Let $x \in \B_n$. There exists a unique decomposition $x = \Delta^p x_1\cdots x_r$ such that $x_1, \ldots,x_r$ are simple braids, distincts from $\Delta$ and $1$, and such that $x_i$ and $x_{i+1}$ are left-weighted for all $i = 1, \ldots, r-1$.
\end{prop}

\begin{defi}[left normal form]
In the previous proposition, the writing $x = \Delta^p x_1 \cdots x_r$ is called the \emph{left normal form} of $x$, $p$ is called the \emph{infimum} of $x$ and is denoted by $\inf x$, $p+r$ is the \emph{supremum} of $x$ and is denoted by $\sup x$, and $r$ is called the \emph{canonical length} of $x$.

Furthermore, we denote by $\iota(x) = \Delta^{-p} x_1 \Delta^p$ the \emph{initial factor} of $x$ ($\iota(x) = x_1$ if $p$ is even, $\iota(x) = \Delta^{-1} x_1 \Delta$ if $p$ is odd), and $\phi(x) = x_r$ its \emph{final factor}.
\end{defi}

\begin{defi}[super summit set]
Let $x \in \B_n$. We call \emph{super summit set} of $x$ (abbreviated SSS) the set of the conjugates of $x$ with the minimal canonical length.
\end{defi}

Obviously, if $x$ and $y$ are conjugate, they have the same super summit set: thus the super summit set is an invariant of the conjugacy class.

\begin{defi}[cycling, decycling]
Let $x \in \B_n$, and let $x = \Delta^p x_1\cdots x_r$ be its normal form. If $r \geqslant 1$, we define
\begin{itemize}
\item the \emph{cycling} of $x$ by $\iota(x)^{-1} x \iota(x) = \Delta_p x_2 \cdots x_r x_1'$,
\item the \emph{decycling} of $x$ by $\phi(x) x \phi(x)^{-1} = \Delta_p x_r' x_1 \cdots x_{r-1}$,
\end{itemize}
where $x_1' = \iota(x) = x_1$ or $\Delta^{-1} x_1 \Delta$, and $x_r' = \phi(x) = x_r$ or $\Delta^{-1}x_r \Delta$, depending on the parity of $p$. If $r=0$, the cycling and the decycling of $x$ are equal to $x$ itself.
\end{defi}

An interesting property of the cycling and decycling operations is that they preserve the SSS. Moreover, from any braid, we can obtain an element of its SSS by applying a finite number of cyclings and decyclings \cite{EM}.

Another type of conjugation is the \emph{cyclic sliding} \cite{GGM}.
\begin{defi}[cyclic sliding]
We call \emph{preferred prefix} of $x$ the greatest common prefix of $\iota(x)$ and $\partial\phi(x) = \phi(x)^{-1}\Delta$. The \emph{cyclic sliding} is defined as the conjugation by the preferred prefix.
\end{defi}

The cyclic sliding also preserves the super summit set.

The operation of cyclic sliding is ultimately periodic, because it does not increase the length of the braid. This allows us to define another invariant of conjugacy classes \cite{GGM}:
\begin{defi}[set of sliding circuits]
We call \emph{set of sliding circuits} of a braid $x$ (abbreviated SC) the set of the conjugates of $x$ that are periodic points of the cyclic sliding.
\end{defi}
The SC of $x$ is a subset of its SSS \cite{GGM}.

\begin{defi}[rigidity]
A braid $x$ is said to be \emph{rigid} if $\phi(x)$ and $\iota(x)$ are left-weighted, that is, if the cyclic sliding of $x$ is equal to $x$.
\end{defi}

A rigid braid, by definition, belongs necessarily to its SC, since it is a periodic point of period $1$ of the cyclic sliding operation.

\section{A family of pseudo-Anosov braids}\label{sec:pA}

Let $k \geqslant 2$ be an integer. Let us define the following braid with $5$ strings:
\[\trak = \delta_3^{3k+1}\sigma_4^{2k+2} \sigma_3\sigma_4^{2k-1}\]
where $\delta_3 = \sigma_2\sigma_1$. Recall that we also have $\Delta_3 = \sigma_2\sigma_1\sigma_2$.

\begin{lem}
The left normal form of $\trak$ is given by the following factorisation
\[\trak = (\Delta_3\sigma_4)^{2k}(\delta_3 \sigma_4 \sigma_3 \sigma_4)(\sigma_3\sigma_4)(\sigma_4)^{2k-3}.\]
In particular, $\inf \trak = 0$ and $\sup \trak = 4k-1$.
\end{lem}

\begin{proof}
We easily check that this writing is indeed equal to $\trak$ (recall that $\Delta_3^2 = \delta_3^3$), that each of these factors is a simple braid, and that two successive factors are left-weighted. (See Figure \ref{fig:trak}.)
\end{proof}

\begin{figure}
\centering
\includegraphics[width=4cm]{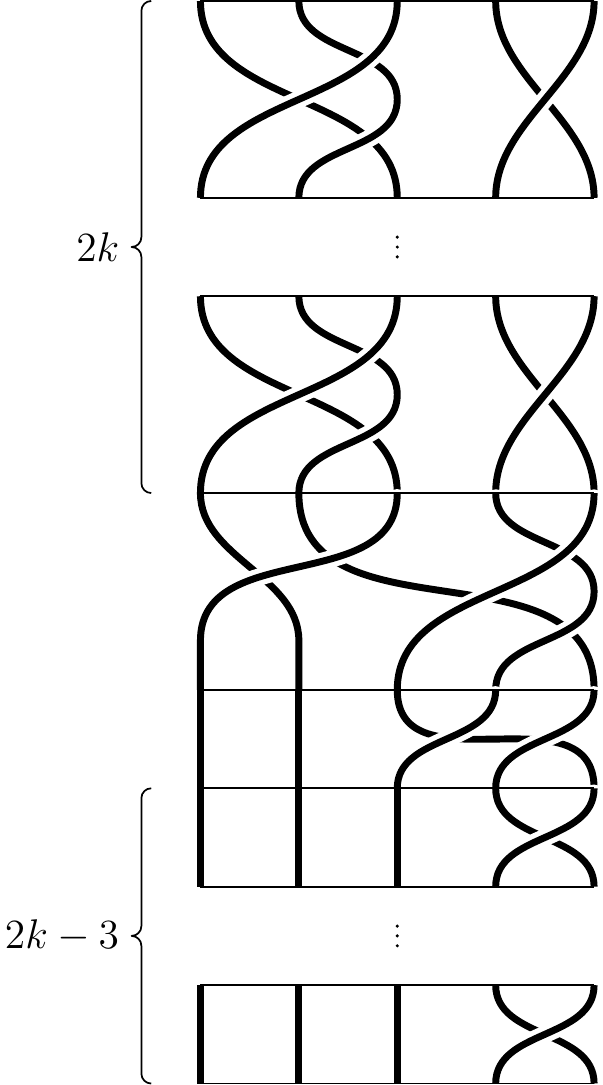}
\caption{Left normal form of the braid $\trak$.}
\label{fig:trak}
\end{figure}

Let us also consider the braid
\[\trbk = (\sigma_1\sigma_3)^{2k-2}(\sigma_3\sigma_4\Delta_3)(\Delta_3\sigma_4)[(\delta_3\sigma_4)(\tdelta_3\sigma_4)]^{k-1}(\delta_3 \sigma_4 \sigma_3 \sigma_4),\]
where $\tdelta_3 = \sigma_1 \sigma_2$. (See Figure \ref{fig:trbk}.)

\begin{figure}
\centering
\includegraphics[width=4cm]{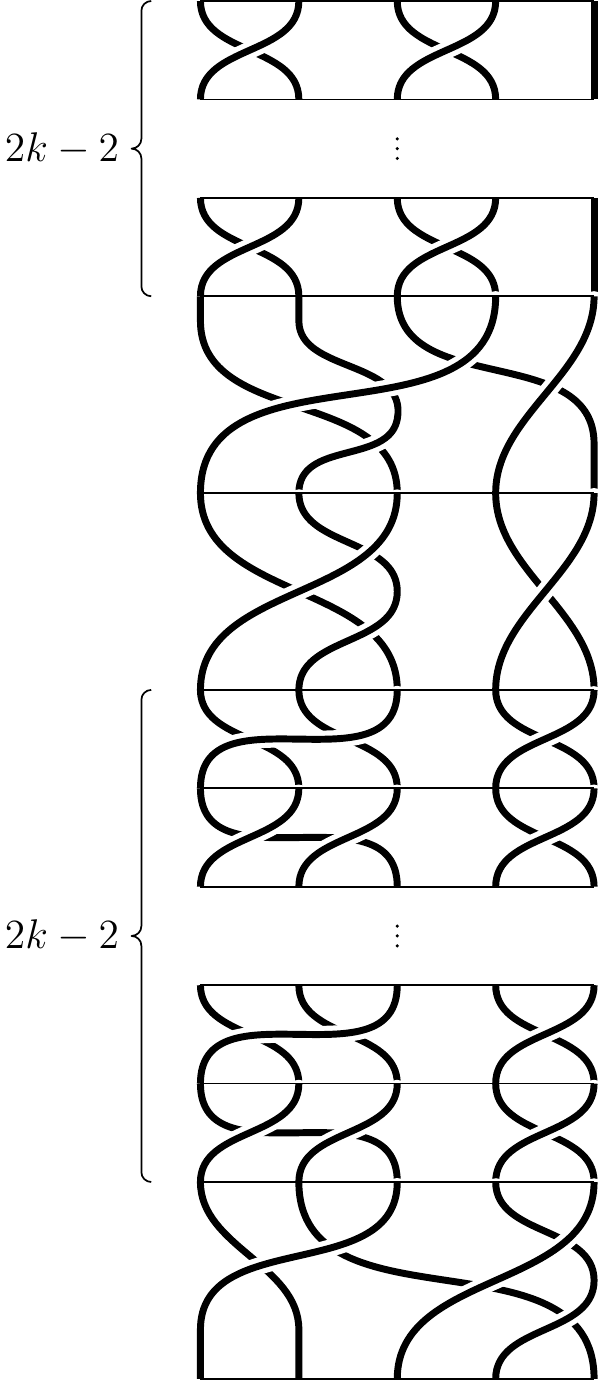}
\caption{The braid $\trbk$ in left normal form}
\label{fig:trbk}
\end{figure}

\begin{lem}\label{lem:trbk}
The braid $\trbk$ is conjugate to $\trak$, and, furthermore, it is rigid and in its SC. Moreover, $\inf \trbk = 0 = \inf \trak$ and $\sup \trbk = 4k-1 = \sup \trak$.
\end{lem}

This immediately implies the
\begin{cor}
The braid $\trak$ belongs to its SSS.
\end{cor}
Indeed, $\trak$ has the same canonical length as one of its conjugates that lies in its SC, and thus in its SSS.

\begin{proof}[Proof of Lemma \ref{lem:trbk}]
The braid $\trbk$ is obtained by conjugating $\trak$ by
\[\tau = (\Delta_3\sigma_4)^{2k}(\delta_3 \sigma_4 \sigma_3 \sigma_4)(\sigma_3\sigma_4\Delta_3)(\Delta_3\sigma_4)^{2k-1}(\delta_3 \sigma_4 \sigma_3 \sigma_4),\]
\emph{i.e.} $\trbk = \tau^{-1} \trak \tau$.
Then, the left normal form of $\trbk$ is
\[\trbk = (\sigma_1\sigma_3)^{2k-2}(\sigma_3\sigma_4\Delta_3)(\Delta_3\sigma_4)[(\delta_3\sigma_4)(\tdelta_3\sigma_4)]^{k-1}(\delta_3 \sigma_4 \sigma_3 \sigma_4).\]
We observe that this braid is rigid, and consequently, belongs to the SC of $\trak$. Furthermore, we have $\inf \trbk = 0$ and $\sup \trbk = 4k-1$.
\end{proof}

\subsection{The SC has only one orbit under cycling or conjugation by $\Delta$}

\begin{lem}\label{lem:orbite}
The SC of $\trak$ is reduced to the orbit of $\trbk$ under the operation of cycling or conjugation by $\Delta$.
\end{lem}
We shall denote this orbit under cycling and conjugation by $\orb(\trbk)$. Note that, because of the rigidity of $\trbk$, a cycling of $\trbk$ is just a cyclic permutation of its factors.

The proof is based on the following two lemmas: Lemma \ref{lem:Geb} gathers some results proved in \cite{Geb} or in \cite{BGGM} (Section 3.3), and Lemma \ref{lem:6.1} is stated and proven in \cite{GMW} (Lemma 6.1).
\begin{defi}[transport]
Let $x \in \B_n$ be in its super summit set. Let $s$ be a simple braid such that $y = s^{-1} x s$ is in the super summit set of $x$. Let $x' = \iota(x)^{-1} x \iota(x)$ and $y' = \iota(y)^{-1} y \iota(y)$ be the braids obtained by cycling from $x$ and $y$, respectively. We call \emph{transport} of $s$ the braid $s^{(1)} = \iota(x)^{-1} s \iota(y)$, that is to say the braid $s^{(1)}$ such that $y' = (s^{(1)})^{-1} x' s^{(1)}$.
\end{defi}
\begin{lem}\label{lem:Geb}
If $s$ is simple, $s^{(1)}$ is simple. Moreover, for $s$ and $t$ simple, $s \preccurlyeq t$ implies $s^{(1)} \preccurlyeq t^{(1)}$, and $(\iota(x))^{(1)} = \iota(x')$. In particular, if $s$ is a prefix of $\iota(x)$, then $s^{(1)}$ is a prefix of $\iota(x')$. We also have $(\partial\phi(x))^{(1)} = \partial\phi(x')$ and in particular, if $s$ is a prefix of $\partial\phi(x)$, then $s^{(1)}$ is a prefix of $\partial\phi(x')$.
\end{lem}
\begin{lem}\label{lem:6.1}
Given two conjugate elements $x$ and $y$, both in their SC, there exists a sequence of elements $x = \alpha_1, \alpha_2, \ldots, \alpha_r, \alpha_{r+1} = y$, all in the SC, and simple braids $s_1, \ldots, s_r$, such that  $\alpha_{i+1} = s_i^{-1} \alpha_i s_i$, $i=1, \ldots, r$, and such that $s_i \preccurlyeq \iota(\alpha_i)$ or $s_i \preccurlyeq \iota(\alpha_i^{-1}) = \partial\phi(\alpha_i)$.
\end{lem}

\begin{proof}[Proof of Lemma \ref{lem:orbite}]
Let $\iota = \iota(\trbk) = \sigma_1\sigma_3$ be the initial factor of $\trbk$ and $\phi = \phi(\trbk) = \delta_3 \sigma_4 \sigma_3 \sigma_4$ its final factor, and $\partial\phi = \sigma_2\sigma_1\sigma_3\sigma_2\sigma_4$ the complement of $\phi$ (that is, $\phi\cdot \partial\phi=\Delta$). 
Let us show that for each strict prefix $p$ of $\iota$ or of $\partial\phi$, the braid $p^{-1}\trbk p$ does not belong to the SC of $\trbk$.

Let us list the strict prefixes of $\iota$ and $\partial\phi$. The strict prefixes of $\iota = \sigma_1\sigma_3$ are $\sigma_1$ and $\sigma_3$. Those of $\partial\phi$ are $\sigma_2$, $\sigma_2\sigma_1$, $\sigma_2\sigma_3$, $\sigma_2\sigma_1\sigma_3$, $\sigma_2\sigma_3\sigma_4$, $\sigma_2\sigma_1\sigma_3\sigma_2$ and $\sigma_2\sigma_1\sigma_3\sigma_4$. For each of them, let us calculate the left normal form of $p^{-1}\trbk p$.

\begin{itemize}
\item For $\sigma_1$ :
\begin{align*}
\sigma_1^{-1} \trbk \sigma_1 &= \sigma_1^{-1}(\sigma_1\sigma_3)^{2k-2}(\sigma_3\sigma_4\Delta_3)(\Delta_3\sigma_4)[(\delta_3\sigma_4)(\tdelta_3\sigma_4)]^{k-1}(\delta_3 \sigma_4 \sigma_3 \sigma_4)\sigma_1\\
&= (\sigma_1\sigma_3)^{2k-2}(\sigma_3\sigma_4\Delta_3)(\tdelta_3\sigma_4)[(\delta_3\sigma_4)(\tdelta_3\sigma_4)]^{k-1}(\delta_3 \sigma_4 \sigma_3 \sigma_4)\sigma_1,\end{align*}
the last line being in left normal form.
\item For $\sigma_3$ :
\[\sigma_3^{-1} \trbk \sigma_3 = (\sigma_1\sigma_3)^{2k-3}(\sigma_3\sigma_4\Delta_3)(\Delta_3\sigma_4)^2(\tdelta_3\sigma_4)[(\delta_3\sigma_4)(\tdelta_3\sigma_4)]^{k-2}(\delta_3 \sigma_4 \sigma_3 \sigma_4)\sigma_3.\]
\item For $\sigma_2$ :
\begin{multline*}
\sigma_2^{-1} \trbk \sigma_2 = \Delta^{-1}(\Delta_3\sigma_3\sigma_2\sigma_4\sigma_3\delta_3)(\sigma_1\sigma_3)^{2k-2}(\sigma_3\sigma_4\Delta_3)(\Delta_3\sigma_4)\\ [(\delta_3\sigma_4)(\tdelta_3\sigma_4)]^{k-1}(\delta_3 \sigma_4 \sigma_3 \sigma_4\sigma_2).\end{multline*}
\item For $\sigma_2\sigma_1$ :
\begin{multline*}
(\sigma_2\sigma_1)^{-1} \trbk \sigma_2\sigma_1 = \Delta^{-1}(\Delta_3\sigma_3\delta_3\sigma_4\sigma_3)(\sigma_1\sigma_3)^{2k-2}(\sigma_3\sigma_4\Delta_3)(\Delta_3\sigma_4)\\ [(\delta_3\sigma_4)(\tdelta_3\sigma_4)]^{k-1}(\delta_3 \sigma_4 \sigma_3 \sigma_4\sigma_2\sigma_1).\end{multline*}
\item For $\sigma_2\sigma_3$ :
\begin{multline*}
(\sigma_2\sigma_3)^{-1} \trbk \sigma_2\sigma_3 = \Delta^{-1}(\tdelta_3\sigma_3\sigma_4\sigma_2\sigma_3\delta_3)(\sigma_1\sigma_3)^{2k-2}(\sigma_3\sigma_4\Delta_3)(\Delta_3\sigma_4)\\ [(\delta_3\sigma_4)(\tdelta_3\sigma_4)]^{k-1}(\delta_3 \sigma_4 \sigma_3 \sigma_4\sigma_2\sigma_3).\end{multline*}
\item For $\sigma_2\sigma_1\sigma_3$ :
\begin{multline*}(\sigma_2\sigma_1\sigma_3)^{-1} \trbk \sigma_2\sigma_1\sigma_3 = \Delta^{-1}(\tdelta_3\sigma_3\delta_3\sigma_4\sigma_3)(\sigma_1\sigma_3)^{2k-2}(\sigma_3\sigma_4\Delta_3)(\Delta_3\sigma_4)\\ [(\delta_3\sigma_4)(\tdelta_3\sigma_4)]^{k-1}(\delta_3 \sigma_4 \sigma_3 \sigma_4\sigma_2\sigma_1\sigma_3).\end{multline*}
\item For $\sigma_2\sigma_3\sigma_4$ :
\begin{multline*}(\sigma_2\sigma_3\sigma_4)^{-1} \trbk \sigma_2\sigma_3\sigma_4 = \Delta^{-1}(\sigma_2\sigma_3\sigma_2\sigma_4\sigma_3\delta_3)(\sigma_1\sigma_3)^{2k-2}(\sigma_3\sigma_4\Delta_3)(\Delta_3\sigma_4)^2\\ (\tdelta_3\sigma_4)[(\delta_3\sigma_4)(\tdelta_3\sigma_4)]^{k-2}(\delta_3 \sigma_4 \sigma_3\sigma_4\sigma_2\sigma_3).\end{multline*}
\item For $\sigma_2\sigma_1\sigma_3\sigma_2$ :
\begin{multline*}(\sigma_2\sigma_1\sigma_3\sigma_2)^{-1} \trbk \sigma_2\sigma_1\sigma_3\sigma_2 = 
(\sigma_3\sigma_4\sigma_3\Delta_3)(\sigma_2\sigma_4)^{2k-2}(\delta_3\sigma_3\sigma_4\sigma_3)\\ (\sigma_1\sigma_4\sigma_3)[(\sigma_1\sigma_3\sigma_4)(\sigma_1\sigma_4\sigma_3)]^{k-1}.\end{multline*}
\item For $\sigma_2\sigma_1\sigma_3\sigma_4$ :
\begin{multline*}(\sigma_2\sigma_1\sigma_3\sigma_4)^{-1} \trbk \sigma_2\sigma_1\sigma_3\sigma_4 = \Delta^{-1}(\sigma_2\sigma_3\delta_3\sigma_4\sigma_3)(\sigma_1\sigma_3)^{2k-2}(\sigma_3\sigma_4\Delta_3)(\Delta_3\sigma_4)\\ [(\delta_3\sigma_4)(\tdelta_3\sigma_4)]^{k-1}(\delta_3 \sigma_4 \sigma_3 \sigma_4\sigma_2\sigma_1\sigma_3\sigma_4).\end{multline*}
\end{itemize}
We observe that the canonical length of $\sigma_1^{-1} \trbk \sigma_1$, $\sigma_2^{-1} \trbk \sigma_2$, $(\sigma_2\sigma_1)^{-1} \trbk \sigma_2\sigma_1$, $(\sigma_2\sigma_3)^{-1} \trbk \sigma_2\sigma_3$, $(\sigma_2\sigma_1\sigma_3)^{-1} \trbk \sigma_2\sigma_1\sigma_3$, $(\sigma_2\sigma_3\sigma_4)^{-1} \trbk \sigma_2\sigma_3\sigma_4$ and $(\sigma_2\sigma_1\sigma_3\sigma_4)^{-1} \trbk \sigma_2\sigma_1\sigma_3\sigma_4$ is $4k$, so that these elements are not in the SSS of $\trbk$, and \emph{a fortiori} not in its SC.

As to $\sigma_3^{-1} \trbk \sigma_3$ and $(\sigma_2\sigma_1\sigma_3\sigma_4)^{-1} \trbk \sigma_2\sigma_1\sigma_3\sigma_4$, they certainly are in the SSS, but they are not rigid, and so they cannot be in the SC, because if the SC contains one rigid element, then all its elements are rigid (see Corollary 11 in \cite{GGM}).

Now, Lemmas \ref{lem:Geb} and \ref{lem:6.1} allow us to conclude. Let us assume that there exists a braid $\alpha$ in the set of sliding circuits of $\trbk$ which is not in $\orb(\trbk)$. According to Lemma \ref{lem:6.1}, there is a sequence $\trbk = \alpha_1, \alpha_2, \ldots, \alpha_r,\alpha_{r+1} = \alpha$ of elements in the set of sliding circuits that verify the conclusions of the Lemma \ref{lem:6.1}. Let $i$ be the smallest index such that $\alpha_1, \ldots, \alpha_{i}$ are in $\orb(\trbk)$ and $\alpha_{i+1}$ is not. We denote by $s$ the braid by which $\alpha_i$ is conjugated to get $\alpha_{i+1}$. According to Lemma \ref{lem:6.1}, it is a simple braid which is a prefix of $\iota(\alpha_i)$ or $\partial\phi(\alpha_i)$.

Let $\alpha_i = \gamma_0, \gamma_1, \ldots, \gamma_t = \trbk$ be the elements of $\orb(\trbk)$ such that $\gamma_{j+1}$ is obtained from $\gamma_j$ by cycling (or, if necessary, $\gamma_{t} = \Delta^{-1}\gamma_{t-1}\Delta$). According to Lemma \ref{lem:Geb}, the transport $s^{(1)}$ of $s$ is a simple braid, prefix of $\iota(\gamma_1)$ or of $\partial\phi(\gamma_1)$. We define by induction $s^{(j+1)} = (s^{(j)})^{(1)}$ (or, if necessary, $s^{(t)} = \Delta^{-1}s^{(t-1)}\Delta$). Still according to Lemma \ref{lem:Geb} (and, if necessary, to the fact that the conjugate by $\Delta$ of a prefix of $\iota(x)$ is a prefix of $\iota(\Delta^{-1}x\Delta)$), for all $j$, $s^{(j)}$ is a prefix of $\iota(\gamma_j)$ or of $\partial\phi(\gamma_j)$. In particular, $s^{(t)}$ is a prefix of $\iota(\trbk)$ or of $\partial\phi(\trbk)$. So $\trbk$ is sent by conjugation by $s^{(t)}$ to a braid that is not in its orbit $\orb(\trbk)$, since it is in the orbit of $\alpha_{i+1}$. But this is impossible: according to the exhaustive case checking above, the conjugate of $\trbk$ by a strict prefix of $\iota(\trbk)$ or of $\partial\phi(\trbk)$ is never an element of the set of sliding circuits, and the conjugates by $\iota(\trbk)$ and $\partial\phi(\trbk)$ themselves are elements of $\orb(\trbk)$.
\end{proof}

\subsection{Round reduction curves}

We want to prove the following theorem
\begin{theo}\label{theo:pseudoAnosov}
The braid $\trbk$ (and thus also the braid $\trak$, which is conjugated to it) is pseudo-Anosov.
\end{theo}
First of all, it is easy to show the
\begin{lem}
The braid $\trbk$ is not periodic.
\end{lem}

\begin{proof}
If $\trbk$ was periodic, one of its powers would be equal to a power of $\Delta$. But $\trbk$ is rigid with $\inf \trbk = 0$, so the left normal form of a power of $\trbk$ is directly obtained by juxtaposing that of $\trbk$ the suitable number of times. It is then obvious that it is not a power of $\Delta$.
\end{proof}

It remains to prove that $\trbk$ is not reducible. The following result (Corollary 4.3 in \cite{GM}) allows us to reduce to the case of round curves (\emph{i.e.} homotopic to a circle):
\begin{prop}\label{prop:GM:courbesrondes}
For every reducible, non periodic braid $x \in \B_n$, there exists a conjugate of $x$ in its SC that sends at least one round curve to a round curve.
\end{prop}
Let us also state the following theorem of Bernadete, Gutierrez and Nitecki (Theorem 5.7 in \cite{BGN}) as given in \cite{Calvez} (Theorem 1):
\begin{prop}\label{prop:BGN}
Let $x \in \B_n$, seen as a mapping class in $\mcg(\disque n, \partial \disque n)$, with left normal form $x = \Delta^p x_1 \cdots x_r$. Let $\mathcal C$ be a round curve in $\disque n$. If $x(\mathcal C)$ is round, then $\Delta^p x_1\cdots x_m(\mathcal C)$ is round for all $m = 1, \ldots, r$.
\end{prop}
A corollary of this last result is that, if an element of the SC sends a round curve to a round curve, then so do the braids obtained by applying cyclings to it (and so do, of course, their conjugates by $\Delta$). Now, according to Lemma \ref{lem:orbite}, the SC is reduced to only one orbit under cycling or conjugation by $\Delta$. Hence it suffices to show that $\trbk$ does not send any round curve to a round curve: this will imply that no element of the SC sends any round curve to a round curve, and then, due to Proposition \ref{prop:GM:courbesrondes}, that $\trbk$ is not reducible.

Figure \ref{fig:courbes} shows all the essential round curves for a braid with $5$ strings.
\begin{figure}[h]
\centering
\includegraphics{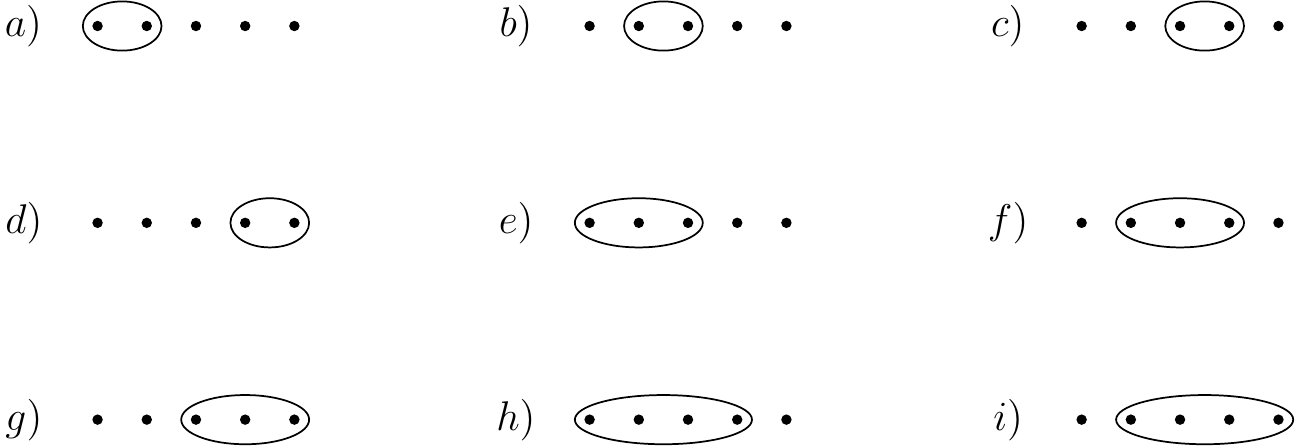}
\caption{All round curves for $5$ strings}
\label{fig:courbes}
\end{figure}

Let us check that none of these curves is sent by $\trbk$ to a round curve. Let us remark that, still according to Proposition \ref{prop:BGN}, it suffices that a braid consisting of the first factors of $\trbk$ transforms the round curve to a curve that is not round, in order to be sure that the final image is not round.

Let us check this explicitly for each of the round curves in Figure \ref{fig:courbes}. After applying the first factor $\sigma_1 \sigma_3$, only the images of the curves $a)$, $c)$, $g)$ and $h)$ are still round. These four are transformed into themselves, and so remain unchanged after applying $(\sigma_1\sigma_3)^{2k-2}$. The factor $\sigma_3\sigma_4\Delta_3$ transforms $a)$ into $b)$, whereas the images of $c)$, $g)$ and $h)$ are no longer  round. By $\Delta_3 \sigma_4$, $b)$ is again transformed into $a)$, then applying the elements $\delta_3\sigma_4$ and $\tdelta_3\sigma_4$ alternatively transforms $a)$ into $b)$ and $b)$ into $a)$. Finally, the last factor $\delta_3\sigma_4\sigma_3\sigma_4$ transforms $a)$ into a curve that is no longer round.

After all, none of the round curves is preserved by $\trbk$. This concludes the proof of Theorem \ref{theo:pseudoAnosov}.

\section{A lower bound for the cardinality of the super summit set}

\begin{theo}\label{theo:exp}
The braid \[\trak = \delta_3^{3k+1}\sigma_4^{2k+2} \sigma_3\sigma_4^{2k-1} = (\Delta_3\sigma_4)^{2k}(\delta_3 \sigma_4 \sigma_3 \sigma_4)(\sigma_3\sigma_4)(\sigma_4)^{2k-3}\]
is pseudo-Anosov, and the cardinality of its super summit set is at least $2^{2k-2}$.\end{theo}

\begin{proof} We have already seen in Section \ref{sec:pA} that the braid is pseudo-Anosov.

Let $\atome 1 1 = \sigma_1$, $\atome 1 2 = \sigma_1\sigma_2$, $\atome 2 1 = \sigma_2 \sigma_1$ and $\atome 2 2 = \sigma_2$, so that for all $i,j \in \{1,2\}$, $\init{\atome i j} = \{\sigma_i\}$ and $\final{\atome i j} = \{\sigma_j\}$, where $\init{x}$ is the set of the generators $\sigma_l$ such that $\sigma_l \preccurlyeq x$, and $\final x$ is the set of the $\sigma_l$ such that $x \succcurlyeq \sigma_l$. If we choose a sequence of elements $i_1, \ldots, i_{2k-2} \in \{1,2\}$, we denote
\[\tra{{k,i_1, \ldots, i_{2k-2}}} = (\atome 1{i_1} \atome{i_1}{i_2} \cdots \atome{i_{2k-3}}{i_{2k-2}})^{-1} \trak (\atome 1{i_1} \atome{i_1}{i_2} \cdots \atome{i_{2k-3}}{i_{2k-2}}).\]
On the one hand, the simple braids $\delta_3 \sigma_4 \sigma_3 \sigma_4$ and $\sigma_3\sigma_4\atome 1{i_1}$ are left-weighted, and this is also the case for $\sigma_3\sigma_4\atome 1{i_1}$ and $\sigma_4\atome {i_1}{i_2}$, for $\sigma_4\atome {i_1}{i_2}$ and $\sigma_4\atome {i_2}{i_3}$, etc. On the other hand, we can calculate the left normal form of $(\atome 1{i_1} \atome{i_1}{i_2} \cdots \atome{i_{2k-3}}{i_{2k-2}})^{-1}$, seen as a braid with $3$ strings: the left normal form of a braid can be easily expressed in terms of its inverse (see \cite{EM}). That of $(\atome 1{i_1} \atome{i_1}{i_2} \cdots \atome{i_{2k-3}}{i_{2k-2}})^{-1}$ in $\B_3$ is
\[\Delta_3^{-(2k-2)}\partial_3^{-(2(2k-2)+1)}(\atome{i_{2k-3}}{i_{2k-2}}) \cdots \partial_3^{-3}(\atome{i_1}{i_2})\partial_3^{-1}(\atome 1{i_1}),\]
where $\partial_3(x) = x^{-1}\Delta_3$ (and so $\partial_3^{-(2l+1)}(x) = \Delta_3^{l+1}x^{-1}\Delta_3^{-l}$). Moreover, $\partial^{-1}(\atome{1}{i_1}) \sigma_4 = \Delta_3 \atome1{i_1} \sigma_4$ and $\delta_3\sigma_4\sigma_3\sigma_4$ are left-weighted. From this, we deduce the left normal form of $\tra{k,i_1, \ldots, i_{2k-2}}$:
\begin{multline*}
\tra{k,i_1, \ldots, i_{2k-2}} = (\Delta_3\sigma_4)^{2}(\partial_3^{-(2(2k-2)+1)}(\atome{i_{2k-3}}{i_{2k-2}})\sigma_4) \cdots (\partial_3^{-3}(\atome{i_1}{i_2})\sigma_4)(\partial_3^{-1}(\atome 1{i_1})\sigma_4)\\ (\delta_3 \sigma_4 \sigma_3 \sigma_4)(\sigma_3\sigma_4\atome 1{i_1})(\sigma_4\atome{i_1}{i_2})\cdots (\sigma_4\atome{i_{2k-3}}{i_{2k-2}}).\end{multline*}
In particular, the canonical length of $\tra{k,i_1, \ldots, i_{2k-2}}$ is $4k-1$, and so this braid is in the super summit set. This normal form also allows us to observe that $\tra{k,i_1, \ldots, i_{2k-2}} = \tra{k,j_1, \ldots, j_{2k-2}}$ if and only if $i_1 = j_1, \ldots, i_{2k-2} = j_{2k-2}$. The $2^{2k-2}$ possible choices for $(i_1, \ldots, i_{2k-2})$ lead to $2^{2k-2}$ distinct elements in the super summit set of $\trak$.
\end{proof}

\end{document}